\documentclass[a4paper,10pt]{amsart}

\usepackage{graphicx}
\usepackage{dirtytalk}
\usepackage{amsthm}
\usepackage{amssymb}
\usepackage{amsmath}
\usepackage{accents}
\usepackage{hyperref}
\usepackage{xcolor}
\usepackage{enumerate}
\usepackage{listings}
\usepackage{complexity}
\usepackage{blkarray}
\usepackage{braket}
\usepackage{lmodern}
\usepackage[english]{babel}
\usepackage{enumitem}

\usepackage[font=small,labelfont=bf]{caption}
\usepackage[font=small,labelfont=normalfont,labelformat=simple]{subcaption}
\graphicspath{{figs/}}

\newtheorem{theorem}{Theorem}[section]

\newtheorem{construction}[theorem]{Construction}
\newtheorem{lemma}[theorem]{Lemma}
\newtheorem{claim}[theorem]{Claim}

\newtheorem{problem}[theorem]{Problem}

\author
{
Raphael Steiner 
}
\thanks{Institute for Operations Research, Department of Mathematics, ETH Z\"{u}rich, Switzerland,  \texttt{raphaelmario.steiner@math.ethz.ch}. The research of the author is funded by the Ambizione Grant No. 216071 of the Swiss National Science Foundation.}

\date{\today}

\title{Fractional chromatic number vs. Hall ratio}

\begin{document}
\maketitle

\begin{abstract}
Given a graph $G$, its \emph{Hall ratio} $\rho(G)=\max_{H\subseteq G}\frac{|V(H)|}{\alpha(H)}$ forms a natural lower bound on its fractional chromatic number $\chi_f(G)$. A recent line of research studied the fundamental question of whether $\chi_f(G)$ can be bounded in terms of a (linear) function of $\rho(G)$. In a breakthrough-result, Dvo\v{r}\'{a}k, Ossona de Mendez and Wu~\cite[\emph{Combinatorica}, 2020]{dvorak} gave a strong negative answer by proving the existence of graphs with bounded Hall ratio and arbitrarily large fractional chromatic number. In this paper, we solve two natural follow-up problems that were raised by Dvo\v{r}\'{a}k et al.

The first problem concerns determining the growth of $g(n)$, defined as the maximum ratio $\frac{\chi_f(G)}{\rho(G)}$ among all $n$-vertex graphs. Dvo\v{r}\'{a}k et al. obtained the bounds $\Omega(\log\log n) \le g(n)\le O(\log n)$, leaving an exponential gap between the lower and upper bound. We almost fully resolve this problem by proving that the truth is close to the upper bound, i.e., $g(n)=(\log n)^{1-o(1)}$. 

The second problem posed by Dvo\v{r}\'{a}k et al. asks for the existence of graphs with bounded Hall ratio, arbitrarily large fractional chromatic number and such that every subgraph contains an independent set that touches a constant fraction of its edges. We affirmatively solve this second problem by showing that such graphs indeed exist.
\end{abstract}

\section{Introduction}
The famous chromatic number $\chi(G)$ of a graph $G$ is among the most fundamental graph parameters, whose importance has motivated the study of several variants as well as lower and upper bounds. One of the most important such variants is the well-known \emph{fractional chromatic number} $\chi_f(G)$ of a graph $G$, which is defined as the optimal value of the LP-relaxation of the natural integer program for the chromatic number. Concretely, $\chi_f(G)$ is defined as the optimal value of the following linear program ($\mathcal{I}(G)$ denotes the collection of independent sets):

\begin{align*}
\tag{P}
    \text{min} \sum_{I \in \mathcal{I}(G)}&{x_I} \\
    \text{s.t.} \sum_{I\in \mathcal{I}(G): v \in I}&{x_I}\ge 1~~(\forall v \in V(G)), \\
    & x_I \ge 0~~(\forall I \in \mathcal{I}(G)).
\end{align*}

There are many other (equivalent) ways of defining $\chi_f(G)$, see e.g.~\cite{scheinerman2011fractional}. One of them is obtained by considering the dual linear program of the above primal program (P), which represents $\chi_f(G)$ as the optimal value of 

\begin{align*}
\tag{D}
    \text{max} &\sum_{v\in V(G)}{w_v} \\
    \text{s.t.}~~~~&\sum_{v \in I}{w_v}\le 1~~(\forall I \in \mathcal{I}(G)), \\
    &~~~~~w_v \ge 0~~(\forall v\in V(G)).
\end{align*}

Given a non-negative vertex-weight function $w:V(G)\rightarrow [0,\infty)$, let us denote by $\alpha_w(G)$ the maximum total weight of an independent set in $G$, that is,
$$\alpha_w(G):=\max_{I \in \mathcal{I(G)}}\sum_{v\in I}w(v).$$

From the new LP-representation (D) of $\chi_f(G)$, one deduces rather easily the following equivalent representation of $\chi_f(G)$:

\begin{equation}
\tag{M} \chi_f(G)=\max_{w:V(G)\rightarrow [0,\infty), w \not\equiv 0}\frac{\sum_{v\in V(G)}{w(v)}}{\alpha_w(G)}.\end{equation}

In other words, $\chi_f(G)$ equals the maximum ratio between the total vertex weight and the largest weight of an independent set over all non-negative weight assignments. If in the above maximum we only consider $\{0,1\}$-weightings of the vertices, then we obtain a natural lower bound on $\chi_f(G)$ that is known as the \emph{Hall ratio} $\rho(G)$, and can be equivalently rewritten as 
$$\rho(G)=\max_{H\subseteq G}\frac{|V(H)|}{\alpha(H)}.$$

So every graph $G$ satisfies $\rho(G)\le \chi_f(G)\le \chi(G)$, and a fundamental question is to understand how tight the above bounds on the fractional chromatic number are. Concerning the upper bound, it is known since a long time that the inequality can be arbitrarily far from tight. In fact, the well-known \emph{Kneser graphs} are examples of graphs for which the fractional chromatic number can be bounded by a constant, and yet the chromatic number can attain arbitrarily large values at the same time, see~\cite{baranyi,lovasz,scheinerman2011fractional}. In contrast, on the lower bound side the relationship between $\rho(G)$ and $\chi_f(G)$ has remained much less understood until recently, and for many examples of graphs (such as the Kneser graphs and binomial random graphs), the Hall ratio $\rho(G)$ does provide a tight (or almost tight) lower bound for $\chi_f(G)$.

This has motivated several researchers to suspect an inverse relationship between $\chi_f(G)$ and $\rho(G)$. For example, Harris~\cite{harris} conjectured in 2016 that there exists an absolute constant $C>0$ such that every graph $G$ satisfies $\chi_f(G)\le C\rho(G)$, and the same problem was implicitly already posed earlier by Johnson Jr.~\cite{johnson2009} in 2009. Notably, at the time this conjecture was posed, it was not even known whether the fractional chromatic number can be bounded by \emph{any} function of the Hall ratio.

In terms of negative results that separate $\chi_f(G)$ and $\rho(G)$, a first step was taken by Johnson Jr.~\cite{johnson2009} who proved that there exist graphs $G$ for which $\chi_f(G)>\rho(G)$. Strengthening this result, Daneshgar et al.~\cite{daneshgar} and Barnett~\cite{barnett} constructed graphs $G$ such that $\frac{\chi_f(G)}{\rho(G)}\ge \frac{6}{5}$ and $\frac{\chi_f(G)}{\rho(G)}\ge \frac{343}{282}$, respectively. 

Note that the Hall ratio of a graph forms an upper bound on its clique number. Thus, a natural place to look for examples of graphs with small Hall ratio and comparably large fractional chromatic number are the classic constructions of triangle-free graphs with large chromatic number. Interestingly, for one of the most famous such constructions, the iterated Mycielskians, Cropper et al.~\cite{cropper} showed that also their Hall ratios grow to infinity. However, even for this explicit construction, it remains open whether the ratio $\frac{\chi_f(G)}{\rho(G)}$ is bounded.

In a recent breakthrough, Blumenthal et al.~\cite{blumenthal} disproved the conjecture of Harris by constructing graphs $G$ for which the ratio $\frac{\chi_f(G)}{\rho(G)}$ becomes arbitrarily large. Independently, Dvo\v{r}\'{a}k et al.~\cite{dvorak} obtained the even stronger result that there exist graphs $G$ for which $\rho(G)$ is bounded by a constant while $\chi_f(G)$ is arbitrarily large, thus showing that $\chi_f(G)$ cannot be upper bounded by a function of $\rho(G)$. 

\begin{theorem}[Dvo\v{r}\'{a}k, Ossona de Mendez, Wu~\cite{dvorak}]\label{thm:dvorak}
For every constant $C>0$ there exists a graph $G$ with $\chi_f(G)\ge C$ and $\rho(G)\le 18$.  
\end{theorem}

While this result resolved the question about a direct qualitative relationship between $\chi_f(G)$ and $\rho(G)$, several interesting related problems remained open. Maybe the most natural next step to take after the result of Dvo\v{r}\'{a}k et al.~stated in Theorem~\ref{thm:dvorak} is to understand how quickly the ratio between $\chi_f(G)$ and $\rho(G)$ can grow as a function of the number of vertices of $G$. This is the first of three problems posed explicitly by Dvo\v{r}\'{a}k et al.~\cite{dvorak}.
\begin{problem}[Problem~4 in~\cite{dvorak}]\label{prob:4}
Determine the function $g:\mathbb{N}\rightarrow \mathbb{R}_+$, where 
$g(n)$ is defined as the maximum of the ratio $\frac{\chi_f(G)}{\rho(G)}$ over all $n$-vertex graphs $G$.
\end{problem}

Dvo\v{r}\'{a}k et al.~\cite{dvorak} observed that a simple greedy argument can be used to show that $g(n)\le O(\log n)$, and that the construction in their proof of Theorem~\ref{thm:dvorak} yields $g(n)\ge \Omega(\log \log n)$. This leaves an exponential gap between the lower and upper bounds on $g(n)$. As the first main contribution of this paper, we close this gap and almost fully resolve Problem~\ref{prob:4} by showing that $g(n)=(\log n)^{1-o(1)}$. 

\begin{theorem}\label{thm:main}
    There exists an absolute constant $C>0$ such that for every sufficiently large integer $n$ there exists an $n$-vertex graph $G$ with $\chi_f(G)\ge \frac{\log n}{C\log \log n}$ and $\rho(G)\le C(\log\log n)^2$. Hence, $g(n)=\Omega\left(\frac{\log n}{(\log \log n)^3}\right)$ and thus $g(n)=(\log n)^{1-o(1)}$. 
\end{theorem}

Another natural problem that is discussed by Dvo\v{r}\'{a}k et al.~\cite{dvorak} is motivated as follows: Theorem~\ref{thm:dvorak} shows that restricting the weight functions in the formula (M) for $\chi_f(G)$ to $\{0,1\}$-functions is in general too restrictive to expect a good approximation of the maximum over all weight-functions. However, one may hope that there exist other natural classes of weight functions that \emph{do} allow for good approximations of the maximum in (M). As a first candidate, Dvo\v{r}\'{a}k et al. proposed to study the class of weight function that are proportional to the degrees of a subgraph. Concretely, given a subgraph $H$ of a graph $G$ let us define a weight function $\mathrm{deg}_H:V(G)\rightarrow \mathbb{N}_0$ by defining $\mathrm{deg}_H(v)$ as the degree of $v$ in $H$ for every $v\in V(H)$ and as $0$ for every $v\in V(G)\setminus V(H)$. Their motivation for studying this class of weight functions came from the insight that at least for the examples they constructed to prove Theorem~\ref{thm:dvorak}, these weight functions would provide good lower bounds for the fractional chromatic number. 

Explicitly, they posed the following open problem.

\begin{problem}[Problem~5 in~\cite{dvorak}]\label{prob:5}
Do there for some constant $c>0$ exist graphs $G$ of arbitrarily large fractional chromatic number such that $\rho(G)\le c$ and $\alpha_{\mathrm{deg}_H}(H)\ge |E(H)|/c$ for every $H\subseteq G$?
\end{problem}

Note that for every subgraph $H$ of a graph $G$, we have that the total weight of $\mathrm{deg}_H$ equals $\sum_{v\in V(G)}\mathrm{deg}_H(v)=2|E(H)|$, which explains the phrasing of the previous problem.

As the second main result of this paper, we affirmatively solve Problem~\ref{prob:5} by proving the following result. 

\begin{theorem}\label{thm:main2}
For every fixed $\delta >0$ and every sufficiently large integer $n$ there exists an $n$-vertex graph $G$ with $\chi_f(G)\ge \frac{\log \log n}{50 \log\log\log n}$, $\rho(G)\le 4+\delta$ and $\alpha_{\mathrm{deg}_H}(G)\ge \frac{|E(H)|}{4+\delta}$ for every subgraph $H\subseteq G$. 
\end{theorem}

Note that in fact, the statement shown by our Theorem~\ref{thm:main2} is slightly stronger than what was asked for by Problem~\ref{prob:5}, in that $\alpha_{\mathrm{deg}_H}(H)\ge \alpha_{\mathrm{deg}_H}(G)$ for every $H\subseteq G$, and this inequality can be strict.

We remark that in order to keep the presentation as simple as possible, we did not optimize the constant $4+\delta$ in the statement of Theorem~\ref{thm:main2}. With a more careful analysis and at the price of a worse lower bound on the fractional chromatic number, one could replace the constant $4+\delta$ by $3+\delta$ (in both places).

Our proofs of both Theorem~\ref{thm:main} and Theorem~\ref{thm:main2} involve a careful analysis of the Hall ratios and subgraphs of a certain class of unbalanced random graph constructions, which generalize and modify a construction recently used by Janzer, Sudakov and the author~\cite{janzer} in a different context.

Another interesting problem about the Hall ratio was recently posed by Walczak at the BIRS Workshop on New Perspectives in Colouring and Structure (September 29--October 4 2024). The problem asks whether for every $\varepsilon>0$ there exist graphs with arbitrarily large fractional chromatic number and Hall ratio at most $2+\varepsilon$. In the problem statement, Walczak also announced the result (joint with Teiki Rigaud) that the so-called \emph{Burling graphs} have arbitrarily large fractional chromatic number and Hall ratio at most $3$, showing that one can take $\varepsilon=1$.  Improving on this, Davies, Hatzel and Yepremyan (personal communication) recently and independently rediscovered the construction used in our proof of Theorem~\ref{thm:main2}, and showed that there exist graphs with Hall ratio at most $3-\delta$ for some explicit $\delta>0$ and arbitrarily large fractional chromatic number. Hence, one can take $\varepsilon<1$ in the above problem.

\medskip

\paragraph*{\textbf{Notation and Terminology.}} Throughout this paper, we use $\log(x)$ to denote the \emph{natural} logarithm of a positive real number $x$. Given a graph $G$, we denote by $V(G)$ its vertex set and by $E(G)$ its edge set. Given a subset $X\subseteq V(G)$, we denote by $G[X]$ the induced subgraph of $G$ with vertex set $X$ and denote $G-X:=G[V(G)\setminus X]$.

\medskip

\paragraph*{\textbf{Organization.}} The rest of the paper is structured as follows. In Section~\ref{sec:frac} we introduce a general kind of random graph construction (Construction~\ref{con:randomgraph}), and prove a general lower bound on the fractional chromatic number of such graphs that holds under very mild conditions on the setup. In Section~\ref{sec:main1} and~\ref{sec:main2} we then prove two probabilistic lemmas that show that two different special cases of this general type of random graph construction exhibit the additional properties required by Theorem~\ref{thm:main} and Theorem~\ref{thm:main2}, respectively.

\section{Fractional chromatic number}\label{sec:frac}

In this section, we bound the fractional chromatic number of a class of unbalanced random graph constructions, defined as follows.

\begin{construction}\label{con:randomgraph}
    Let $\mathcal{B}=(B_i)_{i=1}^{k}$ be a collection of pairwise disjoint sets, ordered such that   $|B_1|\ge \dots\ge |B_k|$. We denote by $G_{\mathcal{B}}$ the random graph with vertex set $V(G_{\mathcal{B}}):=\bigcup_{i=1}^{k}{B_i}$ created as follows: All the sets $B_1,\ldots,B_k$ are independent in $G_{\mathcal{B}}$, and for every pair of vertices $u\in B_i, v\in B_j$ with $i>j$, the edge $uv$ is included in $G_{\mathcal{B}}$ independently with probability $1/|B_j|$.
\end{construction}

Our first lemma below shows that the fractional chromatic number of a random graph as in Construction~\ref{con:randomgraph} with $k$ parts w.h.p. is in $\Omega(k/\log k)$, provided none of the sets $B_i$ is too small\footnote{A special case of this result, for a slight variant of the construction, was already shown in~\cite[Lemma 2.3]{janzer}. However, the proof method that was used for Lemma~2.3 in~\cite{janzer} does not directly transfer to the more general statement formulated in Lemma~\ref{lem:fracchrom}, whose increased generality is needed for the applications in this paper.}. 

\begin{lemma}\label{lem:fracchrom}
Let $\mathcal{B}=(B_i)_{i=1}^{k}$ with $k\ge 3$ be a collection of disjoint sets with $|B_1|\ge \dots\ge |B_k|\ge k$. Then
$$\mathbb{P}\left[\chi_f(G_{\mathcal{B}})\le \frac{k}{10\log k}\right] \le \left(1+k^{-\log k}\right)^k-1=O(k^{1-\log(k)}).$$ In particular, we have $\chi_f(G_{\mathcal{B}})> \frac{k}{10\log k}$ with probability tending to $1$ as $k\rightarrow \infty$.
\end{lemma}
\begin{proof}
Let $\mathcal{B}=(B_i)_{i=1}^{k}$ be a collection of disjoint sets with $|B_1|\ge \dots\ge |B_k|$ and let $G:=G_{\mathcal{B}}$ denote the associated random graph as in Construction~\ref{con:randomgraph}. Let us denote by $w:V(G)\rightarrow \mathbb{R}_+$ the positive weight function defined as $w(v):=\frac{1}{|B_i|}$ for every $v\in B_i$. Then clearly $\sum_{v\in V(G)}{w(v)}=k$, and hence $\chi_f(G)\le \frac{k}{10\log k}$ implies the existence of an independent set $I\subseteq V(G)$ such that $\sum_{v \in I}{w(v)}=\sum_{i=1}^{k}\frac{|I\cap B_i|}{|B_i|}\ge 10\log k$. Let $J\subseteq I$ be defined as the union of all the sets $I\cap B_i$ such that $\frac{|I\cap B_i|}{|B_i|}\ge \frac{\log k}{k}$. We then have $\sum_{i=1}^{k}{\frac{|J\cap B_i|}{|B_i|}}\ge 10\log k-\log k=9\log k$.

Let $\mathcal{T}$ denote the following set of $k$-tuples of integers:

$$\left\{(t_i)_{i=1}^{k}\in \mathbb{N}_0^k\bigg\vert\forall i\in [k]: t_i=0 \text{ or }\frac{\log k}{k}|B_i|\le t_i\le |B_i|, \text{ and }\sum_{i=1}^{k}\frac{t_i}{|B_i|}\ge 9\log k\right\}.$$

For any $(t_1,\ldots,t_k) \in \mathcal{T}$, let us denote by $\mathcal{E}(t_1,\ldots,t_k)$ the random event that there exists an independent set in $G_{\mathcal{B}}$ containing exactly $t_i$ vertices from $B_i$ for every $i\in [k]$. Above we argued that $\chi_f(G)\le \frac{k}{10\log k}$ implies that $\mathcal{E}(t_1,\ldots,t_k)$ occurs for some  $(t_1,\ldots,t_k)\in \mathcal{T}$.
Thus, our goal will be to bound the probabilities of the individual events  $\mathcal{E}(t_1,\ldots,t_k)$ for all $(t_1,\ldots,t_k)\in \mathcal{T}$ and then apply a union bound.
So let us consider $t=(t_1,\ldots,t_k)\in \mathcal{T}$ arbitrary, and let $P=\mathrm{supp}(t):=\{i\in [k]|t_i>0\}$. We then have:

\begin{align*}\mathbb{P}[\mathcal{E}(t_1,\ldots,t_k)]\le & \prod_{i\in P}{\binom{|B_i|}{t_i}}\prod_{(i,j)\in P^2, i<j}\left(1-\frac{1}{|B_j|}\right)^{t_it_j} \\
\le & \prod_{i\in P}\left(\frac{e|B_i|}{t_i}\right)^{t_i} \prod_{(i,j)\in P^2, i<j}\exp(-t_it_j/|B_j|)\\
\le & \prod_{i\in P}\left(\frac{ek}{\log k}\right)^{t_i} \prod_{(i,j)\in P^2, i<j}\exp(-t_it_j/|B_j|) \\
= & \exp\left(\log(ek/\log k)\sum_{i\in P}t_i-\sum_{(i,j)\in P^2, i<j}\frac{t_it_j}{|B_j|}\right)
\end{align*}
Using $|B_1|\ge\dots\ge |B_k|$ and $(t_1,\ldots,t_k)\in \mathcal{T}$, we can estimate as follows:
\begin{align*}&\sum_{(i,j)\in P^2, i<j}\frac{t_it_j}{|B_j|}\ge \sum_{(i,j)\in P^2, i<j}\frac{1}{2}\left(\frac{t_it_j}{|B_i|}+\frac{t_it_j}{|B_j|}\right)\\
=& \frac{1}{2}\sum_{(i,j)\in P^2, i\neq j}\frac{t_it_j}{|B_i|}=\frac{1}{2}\left(\sum_{i\in P}\frac{t_i}{|B_i|}\sum_{i\in P}{t_i}-\sum_{i\in P}\frac{t_i^2}{|B_i|}\right)\\
\ge& \frac{1}{2}\left(\left(\sum_{i\in P}\frac{t_i}{|B_i|}\right)-1\right)\sum_{i\in P}{t_i} \ge \frac{9\log k-1}{2}\sum_{i\in P}{t_i}.
\end{align*}
Plugging into the above, we obtain that
$$\mathbb{P}[\mathcal{E}(t_1,\ldots,t_k)]\le \exp\left(-\left(\frac{9\log k-1}{2}-\log(ek/\log k)\right)\sum_{i\in P}{t_i}\right)$$ $$\le \exp\left(-2\log k\sum_{i\in P}{t_i}\right)$$
$$\le \exp\left(-2\frac{(\log k)^2}{k}\sum_{i\in P}{|B_i|}\right),$$ using $k\ge 3$ in the second to last inequality. Summing this over all $(t_1,\ldots,t_k)\in \mathcal{T}$, clustering the elements of $\mathcal{T}$ according to their support, we obtain:

\begin{align*}
&\mathbb{P}\left[\chi_f(G)\le \frac{k}{10\log k}\right] \le \sum_{t\in \mathcal{T}}\mathbb{P}[\mathcal{E}(t)]\le \sum_{\emptyset\neq P\subseteq [k]}\sum_{\substack{t\in \mathcal{T}, \\ \mathrm{supp}(t)=P}}\exp\left(-2\frac{(\log k)^2}{k}\sum_{i\in P}|B_i|\right) \\
&\le \sum_{\emptyset\neq P\subseteq [k]}\left(\prod_{i\in P}{|B_i|}\right)\cdot \exp\left(-2\frac{(\log k)^2}{k}\sum_{i\in P}|B_i|\right) \\
&=\sum_{\emptyset\neq P\subseteq [k]}\prod_{i\in P}{\exp\left(\log(|B_i|)-2\frac{(\log k)^2}{k}|B_i|\right)}\\
&=-1+\prod_{i=1}^{k}\left(1+\exp\left(\log(|B_i|)-2\frac{(\log k)^2}{k}|B_i|\right)\right).
\end{align*}
Since $|B_i|\ge k$ for all $i\in [k]$ by assumption and since the term $\log(x)-2\frac{(\log k)^2}{k}x$ is monotonically decreasing for $x\ge k$, we obtain the bound
$$\mathbb{P}\left[\chi_f(G)\le \frac{k}{10\log k}\right]\le \left(1+\exp(\log(k)-2(\log k)^2)\right)^k-1\le \left(1+k^{-\log(k)}\right)^k-1,$$
as desired. This concludes the proof of the lemma.
\end{proof}

\section{Proof of Theorem~\ref{thm:main}}\label{sec:main1}

In this section, we prove a lemma (Lemma~\ref{lem:smallish} below) about a special case of Construction~\ref{con:randomgraph} in which the part sizes $|B_i|$ are decreasing exponentially with $i$. The lemma states that with constant probability the Hall ratio of such graphs is bounded by a function growing very slowly with $n$.

In the following, we will omit floor and ceilings from asymptotic expressions when they are not essential (so as to not hinder the flow of reading). 

\begin{lemma}\label{lem:smallish}
Let $n\in \mathbb{N}$,  $k=k(n):=\frac{1}{2}\log_{3}(n)$ and let $B_1,\ldots,B_k$ be pairwise disjoint sets with $|B_i|=n/3^i$ for $i=1,\ldots,k$, where $n$ is sufficiently large. Then $$\mathbb{P}[\rho(G_{\mathcal{B}})\le 300 (\log \log n)^2]>\frac{1}{3}.$$
\end{lemma}
\begin{proof}
For simplicity, let us denote $G:=G_{\mathcal{B}}$ in the following. 
The following claim is at the heart of the idea of the proof. Roughly speaking it states that, with decent probability, any sequence of consecutive sets $B_i$ that is not too long does not contain relatively small subgraphs with too many edges. We then later use this sparsity result to give lower bounds on the independence number of subgraphs of~$G$. For convenience, in the following we set $B_r:=\emptyset$ for all integers $r\in \mathbb{Z}\setminus [k]$.
\begin{claim}\label{cl1} The following statement holds with probability bigger than $1/3$:

For all $(i,j,s)\in \mathbb{Z}^3$ such that $i\in [k]$, $0\le j<\log_2(i)$ and $2|B_{i+1}|<s\le 2|B_{i}|$ there are no subgraphs $H$ of $G$ with $V(H)\subseteq \bigcup_{\ell=i-2^{j+1}+1}^{i-2^j}{B_\ell}$, $|V(H)|\le s$, and $|E(H)|\ge 3s$. \end{claim}
In the rest of the proof, let us refer to the above property/random event as $\mathcal{A}$.
\begin{proof}[Proof of Claim~\ref{cl1}.]
First note that property $\mathcal{A}$ can be equivalently reformulated using the conditions $|V(H)|=s$ and $|E(H)|=3s$ instead of $|V(H)|\le s$ and $|E(H)|\ge 3s$. This is true since we can always go from a graph satisfying the latter conditions to a graph satisfying the first by removing superfluous edges and adding isolated vertices from $\bigcup_{\ell=i-2^{j+1}+1}^{i-2^j}{B_\ell}$ if necessary (this is always possible, since $\left|\bigcup_{\ell=i-2^{j+1}+1}^{i-2^j}{B_\ell}\right|\ge |B_{i-2^j}|>2|B_{i}|\ge s$). 

    For fixed integers $i\in [k]$, $0\le j<\log_2(i)$ and  $2|B_{i+1}|<s\le 2|B_i|$, let us denote by $\mathcal{E}(i,j,s)$ the event that there is a subgraph $H$ of $G$ with $V(H)\subseteq \bigcup_{\ell=i-2^{j+1}+1}^{i-2^j}{B_\ell}$, $|V(H)|=s$, and $|E(H)|=3s$. Note that $i-2^j\ge 1$ and $$\left| \bigcup_{\ell=i-2^{j+1}+1}^{i-2^j}{B_\ell}\right|\le n\sum_{\ell=i-2^{j+1}+1}^{i-2^j}{3^{-\ell}}<\frac{3}{2}n\cdot 3^{-(i-2^{j+1}+1)}=\frac{n}{2}\cdot 3^{-i+2^{j+1}}.$$
    
    We can now bound the probability of $\mathcal{E}(i,j,s)$ via a union bound over all possible placements of the subgraph $H$ inside of $\bigcup_{\ell=i-2^{j+1}+1}^{i-2^j}{B_\ell}$:
    \begin{align*}\mathbb{P}[\mathcal{E}(i,j,s)]\le & \binom{n/2\cdot 3^{-i+2^{j+1}}}{s} \binom{\binom{s}{2}}{3s}\left(\frac{1}{|B_{i-2^{j}}|}\right)^{3s} \\
    \le &\left(\frac{en/2\cdot 3^{-i+2^{j+1}}}{s}\right)^s \left(\frac{es^2/2}{3s}\right)^{3s} \left(\frac{1}{|B_{i-2^{j}}|}\right)^{3s} \\
    = &\left(\frac{e^4}{432}\frac{ns^2\cdot 3^{-i+2^{j+1}}}{|B_{i-2^j}|^3}\right)^s \\
    \le &\left(\frac{e^4}{432}\frac{4n^3\cdot 3^{-2i}\cdot 3^{-i+2^{j+1}}}{n^3\cdot 3^{-3(i-2^j)}}\right)^s\\
    =& \left(\frac{e^4}{108}3^{-2^j}\right)^s<3^{-2^js}.
\end{align*}

Note that the complement event $\overline{\mathcal{A}}$ occurs if and only if there exists some triple of integers $(i,j,s)$ with $i\in [k]$, $0\le j<\log_2(i)$ and $2|B_{i+1}|<s\le 2|B_i|$ such that the event $\mathcal{E}(i,j,s)$ occurs. Further note that since $|B_1|>\cdots>|B_k|$, the choice of $s$ and the condition $2|B_{i+1}|<s\le 2|B_i|$ uniquely determines $i$. Let us, for every $s$, denote this unique index $i$ as $i(s)$. We then find that
\begin{align*}\mathbb{P}[\overline{\mathcal{A}}]\le & \sum_{s=1}^{2|B_1|}\sum_{0\le j<\log_2(i(s))}\mathbb{P}[\mathcal{E}(i(s),j,s)]
\\
\le & \sum_{s=1}^{\infty}\sum_{j=0}^{\infty}3^{-2^js}=0.637\ldots<\frac{2}{3}.
\end{align*}
Hence, $\mathbb{P}[\mathcal{A}]>\frac{1}{3}$, as desired. This concludes the proof of the claim.
\end{proof}
In the rest of the proof, we will argue that if $G$ satisfies property $\mathcal{A}$, then $\rho(G)\le 300 (\log \log n)^2$ for $n$ large enough. Combined with Claim~\ref{cl1}, this will then directly prove that $\mathbb{P}[\rho(G)\le 300(\log\log n)^2]\ge \mathbb{P}[\mathcal{A}]>1/3$.  

So in the following, let us assume that $\mathcal{A}$ holds. We have to show that every non-empty subgraph $H$ of $G$ satisfies $\alpha(H)\ge \frac{|V(H)|}{300 (\log \log n)^2}$. 
So let $H\subseteq G$ be given arbitrarily and let $s:=|V(H)|>0$. Then $0<s\le |V(G)|<n\sum_{i=1}^{\infty}3^{-i}=\frac{n}{2}<2|B_1|$. Hence, there exists a (unique) index $i\in [k]$ such that $2|B_{i+1}|<s\le 2|B_i|$. Let $X:=V(H)\setminus \bigcup_{\ell=i+1}^{k}B_\ell$, and note that
$$|X|\ge s-\sum_{\ell=i+1}^{k}|B_\ell|> s-\frac{3}{2}|B_{i+1}|> s-\frac{3}{2}\frac{s}{2}=\frac{s}{4}.$$ For each integer $0\le j<\log_2(i)$, let us define a subset $X_j$ of $X$ as
$$X_j:=\bigcup_{\ell=i-2^{j+1}+1}^{i-2^j}(X\cap B_\ell),$$

and define $X_{-1}:=X\cap B_i$. Then the sets $X_{-1},\ldots,X_{\lceil\log_2(i)\rceil-1}$ partition $X$. 

We next claim that for every $j=-1,0,\ldots,\lceil\log_2(i)\rceil-1$, we have 
$\alpha(H[X_j])\ge \frac{|X_j|^2}{7s}.$
This is easy to see when $j=-1$, since then by definition $X_{-1}\subseteq B_i$ is an independent set in $G$, and hence we have $\alpha(H[X_{-1}])=|X_{-1}|\ge \frac{|X_{-1}|^2}{7s}$. Now, consider any integer $0\le j<\log_2(i)$. By definition of $s$ and $X_j$ we have $2|B_{i+1}|<s\le 2|B_i|$, $|X_j|\le s$ and $X_j\subseteq \bigcup_{\ell=i-2^{j+1}+1}^{i-2^j}B_\ell$. Hence, the induced subgraph $H[X_j]$ meets all preconditions of the statement $\mathcal{A}$. In particular, our assumption that $\mathcal{A}$ holds now implies that we cannot have $|E(H[X_j])|\ge 3s$, in other words, we must have $|E(H[X_j])|< 3s$. 

Using the well-known bound $\alpha(F)\ge \frac{|V(F)|^2}{2|E(F)|+|V(F)|}$ which holds for every graph $F$, we now obtain that
$$\alpha(H[X_j])\ge  \frac{|X_j|^2}{6s+|X_j|}\ge \frac{|X_j|^2}{7s},$$
as desired. Since $X_{-1}, \ldots,X_{\lceil\log_2(i)\rceil-1}$ partition $X$, there exists some $j$ such that $|X_j|\ge \frac{|X|}{\lceil \log_2(i)\rceil+1}\ge \frac{s}{4(\lceil\log_2(i)\rceil+1)}$. We now obtain
$$\alpha(H)\ge \alpha(H[X_j])\ge \frac{|X_j|^2}{7s}\ge \frac{s}{112(\lceil\log_2(i)\rceil+1)^2}\ge \frac{|V(H)|}{112(\lceil\log_2(k)\rceil+1)^2}.$$
 All in all, this shows that $\mathcal{A}$ implies $$\rho(G)\le 112(\lceil\log_2(k)\rceil+1)^2\le 300 (\log \log n)^2$$ for $n$ large enough. As explained before, this concludes the proof of the lemma.
\end{proof}

Using Lemma~\ref{lem:smallish}, we are now ready to prove our first main result, Theorem~\ref{thm:main}. 
\begin{proof}[Proof of Theorem~\ref{thm:main}]
Set $C:=300$. Let $k=k(n):=\frac{1}{2}\log_3(n)$ and $B_1,\ldots,B_k$ with $|B_i|=\frac{n}{3^i}$ be defined as in Lemma~\ref{lem:smallish}. Note that for every sufficiently large $n$, we then have $|B_1|\ge\cdots\ge|B_k|=\sqrt{n}\gg k$. Hence, the preconditions of Lemma~\ref{lem:fracchrom} are met, which implies that w.h.p. as $n$ (and hence $k=k(n)$) tends to infinity, we have $\chi_f(G_{\mathcal{B}})\ge \frac{k}{10\log k}=\frac{\log n}{20\log(3)\log(\log_3(n)/2)}\ge \frac{\log n}{20\log(3)\log \log n}\ge \frac{\log n}{C\log \log n}$. On the other hand, Lemma~\ref{lem:smallish} implies that $\rho(G_{\mathcal{B}})\le 300(\log \log n)^2=C(\log \log n)^2$ with probability at least $1/3$ for all sufficiently large $n$. Altogether, for $n$ sufficiently large we find that $\chi_f(G_{\mathcal{B}})\ge \frac{\log n}{C\log\log n}$ and $\rho(G_{\mathcal{B}})\le C(\log \log n)^2$ hold simultaneously with positive probability. 

Note further that for every $n$ we have $|V(G_{\mathcal{B}})|<n\sum_{i=1}^{\infty}3^{-i}=\frac{n}{2}$. By the probabilistic method, it follows that for every sufficiently large $n$ there exists a graph $G$ satisfying $\chi_f(G)\ge \frac{\log n}{C\log\log n}$ and $\rho(G)\le C(\log \log n)^2$ with at most $n$ vertices. By adding isolated vertices as necessary, we can w.l.o.g. assume that $G$ has in fact exactly $n$ vertices, since the addition of isolated vertices to a graph leaves both the fractional chromatic number and the Hall ratio unchanged. This establishes the existence of graphs with the desired properties, concluding the proof of Theorem~\ref{thm:main}.
\end{proof}

\section{Proof of Theorem~\ref{thm:main2}}\label{sec:main2}

In this section, we will prove Theorem~\ref{thm:main2}. To do so, we establish Lemma~\ref{lem:small} below, which states that a special case of Construction~\ref{con:randomgraph} in which the part sizes decrease extremely rapidly with $i$ (much faster than the exponential drop used in the construction for Theorem~\ref{thm:main}) guarantees not only a constant bound on the Hall ratio, but also that for every weighting of the vertices of the graph by the degrees of a subgraph, one can find an independent set taking a constant fraction of the total weight. A similar choice of part sizes was previously shown to be useful in the context of avoiding regular subgraphs~\cite{janzer}. As before, in the following we omit floors and ceilings from asymptotic expressions when they are not crucial. 

\begin{lemma}\label{lem:small}
Let $n\in \mathbb{N}$,  $k=k(n):=\frac{1}{3}\log_4(\log n)$ and let $B_1,\ldots,B_k$ be pairwise disjoint sets with $|B_i|=n^{1-4^i\varepsilon}$, where $\varepsilon=\varepsilon(n):=\frac{1}{\sqrt{\log n}}$. Then for every fixed constant $\delta>0$, with probability tending to one as $n\rightarrow \infty$, it holds that 
\begin{enumerate}
    \item[(a)] $\rho(G_{\mathcal{B}})\le 4+\delta$, and
    \item[(b)] for every subgraph $H\subseteq G_{\mathcal{B}}$ there exists an independent set $I$ in $G_{\mathcal{B}}$ that touches at least $\frac{|E(H)|}{4+\delta}$ edges of $H$. 
\end{enumerate}
\end{lemma}
\begin{proof}
In the rest of the proof, let us denote $G:=G_{\mathcal{B}}$. The statements claimed in the lemma will be easy to deduce once we have established that $G$ w.h.p. satisfies two sparsity properties that are isolated in the following claim. As in the proof of Lemma~\ref{lem:smallish}, for convenience let us set $B_r:=\emptyset$ for all $r\in\mathbb{Z}\setminus [k]$. 
\begin{claim}\label{cl2}
W.h.p. as $n\rightarrow \infty$ the following statements hold:
\begin{enumerate}
    \item For every $i\in [k]$ there exists no non-empty subgraph $H$ of $G$ with $|V(H)|\le k^5|B_i|$, $V(H)\subseteq \bigcup_{\ell=1}^{i-1}B_\ell$, and $|E(H)|\ge \frac{3}{2}|V(H)|$. 
    \item For every $i\in [k]$ we have
    $\left|E\left(G\left[\bigcup_{\ell=i+1}^{k}B_\ell\right]\right)\right|\le k^4|B_{i+1}|$.
\end{enumerate}
\end{claim}
\begin{proof}[Proof of Claim~\ref{cl2}]
\noindent
\begin{enumerate}
    \item For $i\in [k]$ and $0<s\le k^5|B_i|$, let us denote by $\mathcal{E}(i,s)$ the event that there exists a subgraph $H$ of $G$ with $V(H)\subseteq \bigcup_{\ell=1}^{i-1}B_\ell$, $|V(H)|=s$ and $|E(H)|\ge \frac{3}{2}s$. As a first step, we will give an upper bound on $\mathbb{P}[\mathcal{E}(i,s)]$. Clearly $\mathcal{E}(i,s)$ cannot be satisfied when $i=1$, so let us assume $i\ge 2$ in the following. Note that given a subgraph $H$ certifying that $\mathcal{E}(i,s)$ holds, we can always delete edges from $H$ to make it have exactly $\lceil \frac{3}{2}s\rceil$ edges.
    Hence, to bound the probability of the event $\mathcal{E}(i,s)$ we may w.l.o.g. take a union bound only over those $H$ with $|V(H)|=s$ and $|E(H)|=\lceil \frac{3}{2}s\rceil$. Note that $$|V(G)|=\sum_{i=1}^{k}{|B_i|}=n\sum_{i=1}^{k}{n^{-4^i\varepsilon}}<n$$ for every sufficiently large $n$. 
    We therefore obtain the following upper bound on the probability of $\mathcal{E}(i,s)$ for every $2\le i\le k$ and $0<s\le k^5|B_i|$:

    \begin{align*}
        \mathbb{P}[\mathcal{E}(i,s)]&\le \binom{|V(G)|}{s}\binom{\binom{s}{2}}{\lceil\frac{3}{2}s\rceil}\left(\frac{1}{|B_{i-1}|}\right)^{\lceil\frac{3}{2}s\rceil}\\
        &\le \left(\frac{en}{s}\right)^s \left(\frac{s^2/2}{\lceil\frac{3}{2}s\rceil}\right)^{\lceil\frac{3}{2}s\rceil}\left(\frac{1}{|B_{i-1}|}\right)^{\lceil\frac{3}{2}s\rceil} \\
        &\le \left(\frac{en}{s}\right)^s \left(\frac{s}{3|B_{i-1}|}\right)^{\lceil\frac{3}{2}s\rceil}\\
        &\le \left(\frac{ens^{1/2}}{(3|B_{i-1}|)^{3/2}}\right)^s\\
        &\le \left(\frac{n(k^5|B_i|)^{1/2}}{|B_{i-1}|^{3/2}}\right)^s\\
        &= \left(k^{5/2}\frac{n\cdot n^{1/2(1-4^i\varepsilon)}}{n^{3/2(1-4^{i-1}\varepsilon)}}\right)^s\\
        &=\left(k^{5/2}n^{-\frac{1}{2}4^{i-1}\varepsilon}\right)^s.
    \end{align*}  
    Since $i\ge 2$ and $\varepsilon=\frac{1}{\sqrt{\log n}}$, we can further bound the last expression to obtain $$\mathbb{P}[\mathcal{E}(i,s)]\le (k^{5/2}n^{-2\varepsilon})^s=(k^{5/2}e^{-2\sqrt{\log n}})^s.$$ Taking a union bound over all possible choices of $i$ and $s$, we now obtain that the probability that statement $(1)$ of the claim does not hold is upper-bounded by 
    \begin{align*}&\sum_{i=1}^{k}\sum_{s=1}^{k^5|B_1|} (k^{5/2}e^{-2\sqrt{\log n}})^s\le k\cdot \sum_{s=1}^{\infty}(k^{5/2}e^{-2\sqrt{\log n}})^s \\
    &=\frac{k^{7/2}e^{-2\sqrt{\log n}}}{1-k^{5/2}e^{-2\sqrt{\log n}}} \rightarrow 0\end{align*} for $n\rightarrow \infty$, where we used that $k=k(n)=O(\log \log n)$ and thus  $k^{7/2}$ (as well as $k^{5/2}$) are asymptotically dominated by 
 $e^{2\sqrt{\log n}}$.
    
    \item The statement in (2) clearly holds for $i=k$, so we may consider only $i\in [k-1]$ in the following. By definition of $G=G_{\mathcal{B}}$, for every pair of indices $\ell<\ell'$ in $[k]$ we have $\mathbb{E}\left[|E(G[B_\ell\cup B_{\ell'}])|\right]=|B_\ell||B_{\ell'}|\cdot \frac{1}{|B_{\ell'}|}=|B_{\ell}|$. From this it directly follows that for every $i\in [k-1]$ we have $$\mathbb{E}\left[\left|E\left(G\left[\bigcup_{\ell=i+1}^{k}B_\ell\right]\right)\right|\right]=\sum_{\ell=i+1}^{k}(k-\ell)|B_\ell|\le k^2|B_{i+1}|.$$
    Using Markov's inequality, this implies that
    $$\mathbb{P}\left[\left|E\left(G\left[\bigcup_{\ell=i+1}^{k}B_\ell\right]\right)\right|> k^4|B_{i+1}|\right]<\frac{k^2|B_{i+1}|}{k^4|B_{i+1}|}=\frac{1}{k^2}.$$
    Using a union bound over all choices of $i$, we obtain that
    \begin{align*}
       &\mathbb{P}\left[\exists i \in [k]: \left|E\left(G\left[\bigcup_{\ell=i+1}^{k}B_\ell\right]\right)\right|> k^4|B_{i+1}|\right]\le k\frac{1}{k^2}=\frac{1}{k(n)}\rightarrow 0 
    \end{align*}
    for $n\rightarrow \infty$. Thus, the statement $(2)$ of the claim indeed holds w.h.p.
\end{enumerate}
\end{proof}
To conclude the proof of the lemma, we will now show that for every large enough $n$, if the statements $(1)$ and $(2)$ of Claim~\ref{cl2} hold, then we also must have that $\rho(G)\le 4+\delta$ (i.e. $\alpha(H)\ge \frac{|V(H)|}{4+\delta}$ for every $H\subseteq G$) and that for every subgraph $H$ of $G$ there exists an independent set in $G$ touching at least $\frac{|E(H)|}{4+\delta}$ edges of $H$. Note that w.l.o.g. it suffices to prove these two statements for all subgraphs $H$ of $G$ without isolated vertices\footnote{This is because we can always add all the isolated vertices of $H$ to a suitable independent set in the subgraph of $H$ induced by the vertices of degree at least one.}. Once this is achieved, the statements of the lemma will be directly implied by Claim~\ref{cl2}.

So suppose in the remainder of the argument that statements $(1)$ and $(2)$ of Claim~\ref{cl2} hold, and consider any non-empty subgraph $H$ of $G$ without isolated vertices. We have to show that $\alpha(H)\ge \frac{|V(H)|}{4+\delta}$ and that $G$ contains an independent set touching at least $\frac{|E(H)|}{4+\delta}$ edges of $H$. To do so, we start by observing
$$|V(H)|\le |V(G)|=\sum_{i=1}^{k}|B_i|\le k|B_1|<k^{5}|B_1|.$$ 
Hence, there exists a (unique) index $i\in [k]$ such that $k^5|B_{i+1}|<|V(H)|\le k^5|B_i|$. Let $X:=V(H)\setminus\bigcup_{\ell=i}^{k} B_\ell$. Then $X\subseteq \bigcup_{\ell=1}^{i-1}B_\ell$, and since $G[X]$ and all of its subgraphs have at most $k^5|B_i|$ vertices, it follows from item (1) of Claim~\ref{cl2} that every subgraph of $G[X]$ has average degree strictly less than $3$. In particular, every subgraph of $G[X]$ contains a vertex of degree at most $2$, which means that $G[X]$ is a $2$-degenerate graph. This implies that $\chi(G[X])\le 3$. Let $X':=V(H)\setminus \bigcup_{\ell=i+1}^{k} B_\ell$. Since $B_i$ is an independent set in $G$ and $X'\subseteq X \cup B_i$, we conclude $\chi(G[X'])\le 4$. Consider a fixed proper $4$-coloring of $G[X']$. Let us denote by $I$ a color class of this coloring of maximum size and by $J$ a color class touching the most edges of $H$. Then clearly $|I|\ge \frac{|X'|}{4}$. Further, denoting $F:=E\left(G\left[\bigcup_{\ell=i+1}^{k}B_\ell\right]\right)$, the definition of $X'$ implies that every edge in $E(H)\setminus F$ has at least one endpoint in $X'$. Hence, $J$ must touch at least $\frac{1}{4}|E(H)\setminus F|$ edges of $H$. To conclude, it remains to  compare the quantities $|X'|$ and $|V(H)|$ as well as $|E(H)\setminus F|$ and $|E(H)|$ to each other: 

By definition of $X'$ we have
$$|X'|\ge |V(H)|-\sum_{\ell=i+1}^{k}|B_\ell|\ge |V(H)|-k|B_{i+1}|,$$
and since $k^5|B_{i+1}|<|V(H)|$ (by choice of $i$) and $k=k(n)\rightarrow \infty$ for $n\rightarrow \infty$, this implies that $|X'|\ge \frac{|V(H)|}{1+\delta/4}$ for $n$ sufficiently large. Similarly, by item $(2)$ of Claim~\ref{cl2} we have
$$|F|=\left|E\left(G\left[\bigcup_{\ell=i+1}^{k}B_\ell\right]\right)\right|\le k^{4}|B_{i+1}|<\frac{1}{k}|V(H)|\le \frac{2}{k}|E(H)|,$$
where we used that $H$ has no isolated vertices by assumption. Hence, for $n$ sufficiently large we also must have $|E(H)\setminus F|\ge \frac{|E(H)|}{1+\delta/4}$. All in all, this implies that
$\alpha(H)\ge |I|\ge \frac{|X'|}{4}\ge \frac{|V(H)|}{4+\delta}$ and that $J$ is an independent set in $G$ touching at least $\frac{|E(H)\setminus F|}{4}\ge \frac{|E(H)|}{4+\delta}$ edges of $H$. This is what had to be shown. As explained previously, this concludes the proof of the lemma.
\end{proof}

Using Lemmas~\ref{lem:fracchrom} and~\ref{lem:small}, we now easily obtain the statement of Theorem~\ref{thm:main2}.

\begin{proof}[Proof of Theorem~\ref{thm:main2}]
Fix $\delta>0$ arbitrarily and let $G=G(\mathcal{B})$ be the random graph as in Lemma~\ref{lem:small}. Then by the latter w.h.p. we have that $\rho(G)\le 4+\delta$ and that for every subgraph $H$ of $G$ there exists an independent set in  $G$ touching at least $\frac{|E(H)|}{4+\delta}$ edges of $H$. Since the number of edges of $H$ touched by an independent set $I$ in $G$ equals the sum $\sum_{v\in I}\mathrm{deg}_H(v)$, this is equivalent to saying $\alpha_{\text{deg}_H}(G)\ge \frac{|E(H)|}{4+\delta}$ for every $H\subseteq G$. 

On the other hand, we have 
$$|B_1|\ge \cdots \ge |B_k|=n^{1-4^k\varepsilon}=n^{1-(\log n)^{-1/6}}>k=\frac{1}{3}\log_4(\log n)$$ for every sufficiently large $n$. Thus, the preconditions of Lemma~\ref{lem:fracchrom} are met, and we obtain that with high probability as $n\rightarrow\infty$ (and hence $k\rightarrow \infty$) we have $$\chi_f(G)\ge \frac{k}{10\log k}\ge \frac{\log \log n}{30\log(4) \log \log \log n}\ge \frac{\log \log n}{50\log \log\log n}.$$

Summarizing, w.h.p. as $n\rightarrow \infty$ the graph $G$ has at most
$$\sum_{i=1}^{k}n^{1-4^i\varepsilon}<n$$ vertices and satisfies $\rho(G)\le 4+\delta$, $\alpha_{\text{deg}_H}(G)\ge \frac{|E(H)|}{4+\delta}$ for every $H\subseteq G$ as well as $\chi_f(G)\ge \frac{\log \log n}{50\log \log \log n}$. Since the latter three properties are maintained by adding to $G$ a set of $n-|V(G)|$ isolated vertices, this implies the assertion of Theorem~\ref{thm:main2}. 
\end{proof}

\paragraph*{\textbf{Acknowledgements.}} The author would like to thank Barnab\'{a}s Janzer, James Davies, Meike Hatzel and Liana Yepremyan for discussions on the topic.
\bibliographystyle{abbrvurl}
\bibliography{references}

\end{document}